\documentclass[11pt]{amsart}
\usepackage{amssymb}
\usepackage{amsmath,amscd}
\usepackage{combelow}
\usepackage{mathtools}
\usepackage{xcolor}
\usepackage{hyperref}
\usepackage{geometry}
 \usepackage{enumitem}% http://ctan.org/pkg/enumitem

\newcommand{\R}{\mathbb{R}}

\newcommand{\Z}{\mathbb{Z}}
\newcommand{\C}{\mathbb{C}}
\newcommand{\D}{\mathbb{D}}

\newcommand{\Ham}{\operatorname{Ham}}

\newcommand{\lk}{\operatorname{lk}}

\DeclarePairedDelimiter{\norm}{\lVert}{\rVert}
\DeclarePairedDelimiter{\abs}{\lvert}{\rvert}

\newtheorem{theorem}{Theorem}

\newtheorem*{question*}{Question}
\newtheorem{definition}[theorem]{Definition}

\newtheorem{prop}[theorem]{Proposition}

\newtheorem*{lemma*}{Lemma}
\newtheorem*{theorem*}{Theorem}
\newtheorem*{remark*}{Remark}
\newtheorem*{definition*}{Definition}
\newtheorem{remark}[theorem]{Remark}
\theoremstyle{remark}

\newtheorem*{remarks*}{Remarks}

\theoremstyle{definition}

\newtheorem*{claim*}{Claim}

\newtheorem*{example*}{Examples}

\title{Hofer Distance on Lagrangian Links inside the Disc}
\author{Ibrahim Trifa}

\begin{document}

\begin{abstract}
    We show that the set of Hamiltonian isotopies of certain unions of circles inside the disc is unbounded for the Hofer distance. The proof relies on a result by Francesco Morabito \cite{M} together with a standard argument of Michael Khanevsky \cite{K11}.
\end{abstract}

\maketitle

\section{Introduction}

Let $(M,\omega)$ be a symplectic manifold, and $L_0$ be a Lagrangian submanifold. Consider the set $\mathcal L(L_0,M,\omega):= \{\varphi(L_0),\varphi\in\Ham_c(M,\omega)\}$, where $\Ham_c(M,\omega)$ denotes the group of Hamiltonian diffeomorphisms that are compactly supported in the interior of $M$.
This set can be equipped with the Hofer distance:
$$d_H(L,L')=\inf\{\norm{\varphi}|\varphi\in\Ham_c(M,\omega),\varphi(L) = L'\}$$
where $\norm{\varphi}$ denotes the usual Hofer norm of a Hamiltonian diffeomorphism, defined in Section \ref{sec:Khanevsky}.
We know very little about this distance. An important question is whether it is unbounded or not. For instance when $L_0$ is the equator of the two-sphere, this has been an open question for more than thirty years (\cite[Problem 32]{MS}). It is also open when $L_0$ is a circle inside the unit disc, which is what motivated us to study unions of circle inside the disc.

One can show that the Hofer distance is bounded when $L_0$ is the unit circle inside the plane $\R^2$ equipped with its standard symplectic form. Indeed, for any $L'\in \mathcal L(L_0,\R^2,\omega_{std})$, we can find $L''\in \mathcal L$ disjoint from both $L$ and $L'$. Then we get that $d_H(L,L')\leq d_H(L,L'')+d_H(L'',L')\leq 2\pi$, since it only costs the area of a disc to displace it to a disjoint one of the same area.
In some other cases, the Hofer distance is unbounded: for instance Khanevsky proved it when $L_0$ is a diameter inside the unit disc in $\R^2$ \cite{K09}, a non-contractible circle inside the cylinder $S^1\times[0,1]$, or a non-displaceable contractible circle inside this same cylinder \cite{K11}. Khanevsky proved all those results using the same strategy that we will explain in Section \ref{sec:Khanevsky}. It relies on the existence of some quasimorphisms on the group of Hamiltonian diffeomorphisms, which in his case come from Entov and Polterovich's construction \cite{EP}. In 2013, Sobhan Seyfaddini generalised Khanevsky's first result to the case of the standard Lagrangian in a Euclidian ball of any even dimension \cite{S}.
As far as we know, this distance has never been proven to be bounded for any monotone Lagrangian submanifolds.

In this paper, we consider the following case: let $(\D,\omega)$ denote the closed unit disc in $\C$ equipped with the standard symplectic form, normalised so that the total area of the disc is $1$. Let $\Ham_c(\D,\omega)$ be the group of Hamiltonian diffeomorphisms of $(\D,\omega)$ supported in the interior of $\D$. Let $k\geq 2$, and $\underline L_0$ be a disjoint union of $k$ embedded smooth closed simple curves bounding discs of the same area $A>\frac 1 {k+1}$.
Then,
\begin{theorem}
\label{thm:unbdd}
    $d_H$ is unbounded on $\mathcal L(\underline L_0,\D,\omega)$.
\end{theorem}
The proof relies on the same strategy as Khanevsky, together with a result of Morabito. We explain those in the next Section, then move on to the proof of the main result in Section \ref{sec:proof}.

\subsection*{Acknowledgements}
The author would like to thank his PhD advisor Sobhan Seyfaddini for his support, as well as Patricia Dietzsch, Baptiste Serraille and Francesco Morabito for the insightful discussions about Khanevsky's results.

\section{Preliminaries}
\subsection{Hofer distance, Quasimorphisms and Khanevsky's argument}
\label{sec:Khanevsky}
We start by fixing notations and recalling the definition of the Hofer distance.

Let $(M,\omega)$ be a symplectic manifold. A \textit{Hamiltonian} on $M$ is a smooth function $H:S^1\times M\to R$. Any Hamiltonian induces a \textit{Hamiltonian vector field} $X_{H_t}$ uniquely defined by the equation $$\omega(X_{H_t},\cdot)=-dH_t$$
The time 1 map of the flow $\phi^t_H$ of $X_{H_t}$ is called the \textit{Hamiltonian diffeomorphism} generated by $H$.
We denote by $\Ham(M,\omega)$ the set of all Hamiltonian diffeomorphisms of $(M,\omega)$; it is in fact a group.
It can be equipped with the \textit{Hofer norm}, defined by:
$$\norm{\varphi}=\inf\limits_{H,\varphi=\phi^1_H} \int_{S^1}(\max\limits_M H_t-\min\limits_M H_t)dt$$

Before stating Khanevsky's argument, we need to recall the definition of a quasimorphism.
\begin{definition}
    Let $G$ be a group. A quasimorphism of $G$ is a map $q:G\to\R$ satisfying:
    $$\exists D>0,\forall g,h \in G, \abs{q(gh)-q(g)-q(h)}\leq D$$
    The infimum of all $D$ such that this property is satisfied is called the defect of $q$.
    Moreover, a quasimorphism $q$ is said to be homogeneous if for any $g$ in $G$ and integer $n$ in $\Z$, $q(g^n)=nq(g)$. 
    Given any quasimorphism, one can always homogenize it:
    Let $q:G\to\R$ be a quasimorphism. Define $\mu:G\to\R$ by the formula
    $$\mu(g) = \lim\limits_{n\to\infty}\frac {q(g^n)} n$$
    Then, $\mu$ is a well defined homogeneous quasimorphism.
\end{definition}

We are now ready to state Khanevsky's argument. Let $L_0$ be a Lagrangian submanifold of $(M,\omega)$. Denote by $\Ham_c(M,\omega)$ the subgroup of Hamiltonian diffeomorphisms that are compactly supported in the interior of $M$, and let $\mathcal L(L_0,M,\omega):=\{\varphi(L_0),\varphi\in\Ham_c(M,\omega)\}$. It is equipped with the Hofer distance:
$$d_H(L,L')=\inf\{\norm{\varphi}|\varphi\in\Ham(M,\omega),\varphi(L) = L'\}$$
Denote by $\mathcal S(L_0,M,\omega)$ the stabiliser of $L_0$ in $\Ham_c(M,\omega)$, i.e. the subgroup consisting of diffeomorphisms $\varphi$ such that $\varphi(L_0) = L_0$. Then, we have the following (\cite[Proposition 7.4]{K11}):

\begin{prop}[Khanevsky]
\label{prop:Khanevsky}
    Suppose there exists a non-vanishing homogeneous quasimorphism  $r$ on $Ham_c(M,\omega)$, Lipschitz with respect to the Hofer distance, and that vanishes on $\mathcal S(L_0,M,\omega)$. Then, $\mathcal L(L_0,M,\omega)$ has infinite diameter for the Hofer distance.
\end{prop}

\begin{proof}
    Let $\varphi\in\Ham_c(M,\omega)$ be such that $r(\varphi)\neq 0$. Let $D$ be the defect of $r$. For $n$ a positive integer, let $L_n:=\varphi^n(L_0)$.
    Then, $d_h(L_0,L_n)=\inf\{\norm{\psi}|\psi(L_0)=L_n\}$. Let $\psi\in\Ham(M,\omega)$ be such that $\psi(L_0)=L_n=\varphi^n(L_0)$. Then, $\varphi^{-n}\psi(L_0)=L_0$, i.e. $\varphi^{-n}\psi\in\mathcal S(L_0,M,\omega)$. Therefore, $r(\varphi^{-n}\psi)=0$ and:
    \begin{align*}
        \norm{\psi}&\geq \lvert r(\psi) \rvert \text{ since $r$ is Hofer-Lipschitz}\\
        &\geq \lvert r(\psi)-r(\varphi^{-n}\psi)\rvert\\
        &\geq \lvert r(\varphi^{-n})\rvert - \lvert r(\varphi^{-n}+r(\psi)-r(\varphi^{-n}\psi)\rvert\\
        &\geq n \lvert r(\varphi)\rvert - D \text{ since $r$ is a homogeneous quasimorphism of defect $D$}
    \end{align*}
    Taking the infimum over $\psi$, we get $d_H(L_0,L_n)\geq n \lvert r(\varphi)\rvert - D$. Since $r(\varphi)\neq 0$, it shows that $\mathcal L(L_0,M,\omega)$ has infinite diameter.
\end{proof}

\subsection{Morabito's result}
\label{sec:Morabito}

Let $(\D,\omega)$ denote the closed unit disc in $\C$ equipped with the standard symplectic form, normalised so that the total area of the disc is $1$. Let $\Ham_c(\D,\omega)$ be the group of Hamiltonian diffeomorphisms of $(\D,\omega)$ supported in the interior of $\D$. Let $k\geq 2$, and $\underline L_0$ be a disjoint union of $k$ embedded smooth closed simple curves $(L_0^i)_{1\leq i\leq k}$ bounding discs of the same area $A>\frac 1 {k+1}$.

Let $\varphi$ be a Hamiltonian diffeomorphism in $\mathcal S(\underline L_0, \D, \omega)$. Then, there exists a permutation $\sigma\in \mathfrak S_k$ such that for $1\leq i \leq k$, $\varphi(L_0^i)=L_0^{\sigma(i)}$. By choosing a Hamiltonian isotopy $\varphi_t$ from the identity to $\varphi$, one can associate a braid with $k$ strands to $\varphi$. This construction does not depend on the choice of isotopy, and therefore defines a map $b$ from $\mathcal S(\underline L_0, \D, \omega)$ to $\mathcal B_k$, the group of braids with $k$ strands.

For $i=1,2$, consider a symplectic embedding $\Phi_i$ of the disc $(\D,\omega)$ into a two-sphere $\mathbb S_i$ of area $1+s_i$, where $s_1,s_2$ are two different points of the interval $(0,(k+1)A-1]$.

Denote by $\underline L_{0,i}$ the image of $\underline L_0$ by $\Phi_i$. Then, $\underline L_{0,i}$ is an $\eta_i$-monotone Lagrangian link on $\mathbb S_i$ (in the sense of \cite[Definition 1.12]{CGHMSS22}), where $\eta_i=\frac {(k+1)A-1-s_i}{2(k-1)}$.

Therefore, the construction in \cite{CGHMSS22} provides us with a well defined spectral invariant $c_{\underline L_{0,i}}:\widetilde{\Ham(\mathbb S_i)}\to\R$, whose pullback by $\Phi_i$ descends to a Hofer-Lipschitz quasimorphism on $\Ham_c(\D,\omega)$, that we still denote by $c_{\underline L_{0,i}}$. Let $\mu_i$ be its homogenization.

Then, Morabito proves the following in \cite{M}:

\begin{theorem}[Morabito]
\label{thm:Morabito}
    For any $\varphi$ in $\mathcal S(\underline L_0,\D,\omega)$, we have:
    $$\mu_1(\varphi)-\mu_2(\varphi)=\frac{\eta_2-\eta_1}{2k}\lk(b(\varphi))$$
    where $\lk$ is the linking number of a braid.
\end{theorem}

\section{Proof of the main result}
\label{sec:proof}
Putting together Khanevsky's statement and Morabito's result, we can now give a proof of Theorem \ref{thm:unbdd}.

\begin{proof}[Proof of Theorem \ref{thm:unbdd}]
    We want to prove that $d_H$ is unbounded on $\mathcal L(\underline L_0,\D,\omega)$.
    By Khanevsky's statement (Proposition \ref{prop:Khanevsky}), we only need to construct a non-vanishing, homogeneous, Hofer-Lipschitz quasimorphism $r:\Ham_c(\D,\omega)\to\R$ which vanishes on $\mathcal S(\underline L_0,\D,\omega)$. 

    For $i=1,...,4$, consider a symplectic embedding $\Phi_i$ of the disc $(\D,\omega)$ into a sphere $\mathbb S_i$ of area $1+s_i$, where $s_1,...,s_4$ are four different points of the interval $(0,(k+1)A-1]$.

    Following the same construction as in Section \ref{sec:Morabito}, we end up with four Hofer-Lipschitz, homogeneous quasimorphisms $\mu_1,...,\mu_4$ on $\Ham_c(\D,\omega)$.

    Then, by Morabito's Theorem \ref{thm:Morabito}, for $i\neq j$, and $\varphi\in\mathcal S(\underline L_0,\D,\omega)$, we have:
    $$\mu_i(\varphi)-\mu_j(\varphi)=\frac{\eta_j-\eta_i}{2k}\lk(b(\varphi))$$
    
    Let $r:=(\eta_4-\eta_3)(\mu_1-\mu_2)-(\eta_2-\eta_1)(\mu_3-\mu_4)$.
    Then, $r$ is a homogeneous, Hofer-Lipschitz quasimorphism on $\Ham_c(\D,\omega)$, which vanishes on $\mathcal S$.
    
    It only remains to show that $r$ is not identically zero.

    By \cite[Theorem 7.7 (i)]{CGHMSS22}, the homogenized quasimorphisms defined using $\eta$-monotone links only depend on the number of components and the constant $\eta$. Therefore, $r=(\eta_4-\eta_3)(\mu_{\underline L'_{1}}-\mu_{\underline L'_{2}})-(\eta_2-\eta_1)(\mu_{\underline L'_{3}}-\mu_{\underline L'_{4}})$, where, for $i=1,...,4$, $\underline L'_{i}$ is any choice of $\eta_i$-monotone link with $k$ components in $\mathbb S_i$.
    
    We now present such a choice. For $a\in(0,1]$, let $C_a\subset \D$ be the circle of radius $\sqrt a$ centered at the origin (so that it bounds a disc of area $a$).
    Then, we can complete the images of the two concentric circles $C_A$ and $C_{2A-2\eta_i}$ by the embedding $\Phi_i$ into an $\eta_i$-monotone link with $k$ (nested) components $\underline L'_i$.
    Since the $C_{2A-2\eta_i}$, $i=1,...,4$ are disjoint, we can construct a Hamiltonian $H$ on $\D$, equal to $1$ on $C_{2A-2\eta_1}$, and supported in a small neighbourhood of it so that by Lagrangian control (\cite[Theorem 7.6]{CGHMSS22}), $r(\varphi_H)=\frac {\eta_4-\eta_3}{k}\neq 0$.
    
\end{proof}

\begin{remark}
    The reason we need at least two circle components in our link to prove that $d_H$ is unbounded is because we want to be able to embed the disc into spheres of different areas to get $\eta$-monotone links with different parameters $\eta$. If we have a single circle $L_0$ inside the disc bounding a disc of area $A>\frac 1 2$, then the only way for it to be monotone after embedding the disc into a sphere is to choose a sphere of area $2A$, and therefore our strategy cannot produce a quasimorphism satisfying Khanevsky's conditions.
    
    Another idea could be to embed the disc into spheres $\mathbb S_i$ of area $(i+1)A$ for different integers $i\geq 2$, then consider monotone links $\underline L_i$ consisting of the image of $L_0$ and $(i-1)$ parallel circles such that the area of each connected component of the complement of $L_i$ is equal to $A$. Then, the quasimorphisms $\mu_i := \mu_{\underline L_i}$ would be Hofer-Lipschitz homogeneous quasimorphisms on $\Ham_c(\D,\omega)$, and one could easily build linear combinations that vanish on $\mathcal S(L_0,\D,\omega)$. However, we conjecture that such a linear combination is always identically zero, and therefore does not produce a quasimorphism satisfying Khanevsky's conditions. This is related to \cite[Question 4.2]{BFPS} about whether some linear combinations of quasimorphisms defined on $\Ham(S^2)$ identically vanish or not.
\end{remark}

%%%%%%%%%%%%%%%%%%%%%%%%%%%%%%%%%        Bibliography      %%%%%%%%%%%%%%%%%%%%%%%%%%%%%%%%%%%%%%%%%%%%% 
 \bibliographystyle{alpha}
 \bibliography{biblio}

\end{document}